\documentclass[3p]{elsarticle}
 \usepackage{graphics}
 \usepackage{graphicx}
 \usepackage{epsfig}
\usepackage{amssymb}
 \usepackage{amsthm}
 \usepackage{lineno}
 \usepackage{amsmath}
   \numberwithin{equation}{section}
\usepackage{mathrsfs}
\usepackage{color}

\theoremstyle{plain}
\newtheorem{theorem}{Theorem}

\newtheorem{lemma}[theorem]{Lemma}

\newtheorem{definition}[theorem]{Definition}

\numberwithin{equation}{section} \numberwithin{theorem}{section}
\newproof{pot}{{\bf{Proof of Theorem \ref{Thm1.1}}\rm}}

\begin{document}
\begin{frontmatter}
\author[a]{Xiaojun Chang\footnotemark[2]}
\ead{changxj100@nenu.edu.cn}
\author[b,c]{Vicen$\c{t}$iu D. R\u{a}dulescu\footnotemark[2]}
\ead{radulescu@inf.ucv.ro}
\author[a]{Ru Wang\footnotemark[2]}
\ead{wangr076@nenu.edu.cn}
\author[d]{Duokui Yan\footnotemark[1]\footnotemark[2]}
\ead{duokuiyan@buaa.edu.cn}

\address[a]{School of Mathematics and Statistics $\&$ Center for Mathematics and Interdisciplinary Sciences,\\
 Northeast Normal University, Changchun 130024, Jilin,
 PR China}
\address[b]{Faculty of Applied Mathematics, AGH University of Science and Technology, 30-059 Krak\'ow, Poland}
\address[c]{Simion Stoilow Institute of Mathematics of the Romanian Academy, 21 Calea Grivi$\c{t}$ei Street, 010702 Bucharest, Romania}
\address[d]{School of Mathematical Sciences,
Beihang University, Beijing 100191,
PR China}
\footnotetext{\footnotemark[1]{Corresponding author.}}
\footnotetext{\footnotemark[2]{These authors contributed equally to this work.}}

\title{Convergence of least energy sign-changing solutions for logarithmic Schr\"{o}dinger equations on locally finite graphs}

\begin{abstract}
 In this paper, we study the following logarithmic Schr\"{o}dinger equation
\[
-\Delta u+\lambda a(x)u=u\log u^2\ \ \ \ \mbox{ in }V
\]
on a connected locally finite graph $G=(V,E)$, where $\Delta$ denotes the graph Laplacian, $\lambda > 0$ is a constant, and $a(x) \geq 0$ represents the potential. Using variational techniques in combination with the Nehari manifold method based on directional derivative, we can prove that, there exists a constant $\lambda_0>0$ such that for all $\lambda\geq\lambda_0$, the above problem admits a least energy sign-changing solution $u_{\lambda}$. Moreover, as $\lambda\to+\infty$, we prove that the solution $u_{\lambda}$ converges to a least energy sign-changing solution of the following Dirichlet problem
\[\begin{cases}
-\Delta u=u\log u^2~~~&\mbox{ in }\Omega,\\
u(x)=0~~~&\mbox{ on }\partial\Omega,
\end{cases}\]
where $\Omega=\{x\in V: a(x)=0\}$ is the potential well.
\end{abstract}
\end{frontmatter}

\section{Introduction and main results}
Theory of network (or graph) has a wide range of applications in various fields such as signal processing, image processing, data clustering and machine learning. (For example, see \cite{EDLL2014,EDT2017,LE2017}.) A graph $G=(V,E)$, where $V$ denotes the vertex set and $E$ denotes the edge set, is said to be locally finite if for any $x\in V$, there are only finite $y\in V$ such that $xy\in E$. A graph is connected if any two vertices $x$ and $y$ can be connected via finite edges. For any $xy\in E$, we assume that its weight $\omega_{xy}>0$ and $\omega_{xy}=\omega_{yx}$. The degree of $x\in V$ is defined by $deg(x)=\sum_{y\sim x}\omega_{xy}$, where we write $y\sim x$ if $xy\in E$. The distance $d(x,y)$ of two vertices $x,y\in V$ is defined by the minimal number of edges which connect these two vertices. The measure $\mu:V\to\mathbb{R}^+$ is defined to be a finite positive function on $G$.

In recent years, there have been many studies on the existence and multiplicity of solutions to nonlinear elliptic equations on discrete graphs, we refer to \cite{Ge2018-2,Grigor'yan2016-1,Grigor'yan2016-2,Grigor'yan2017,LZ2019,LY2022,XZ2021,Zhang2018} and their references.
In \cite{Grigor'yan2017}, Grigor'yan, Lin and Yang studied nonlinear Schr\"odinger equations
\begin{equation}\label{eq-f}
  -\Delta u+b(x)u=f(x,u) \quad \mbox{in }V
\end{equation}
on a connected locally finite graph $G$. By applying the mountain pass theorem, they established the existence of strictly positive solutions of (\ref{eq-f}) when $f$ satisfies the so-called Ambrosetti-Rabinowitz ((AR) for short) condition, and the potential $b: V\to \mathbb{R}^+$ has a positive lower bound and satisfies one of the following hypotheses:
\begin{description}
\item[$(B_1)$]$b(x)\to+\infty$ as $d(x,x_0)\to+\infty$ for some fixed $x_0\in V$;
\item[$(B_2)$]$1/b(x)\in L^1(V)$.
\end{description}
In \cite{Zhang2018}, Zhang and Zhao established the existence and convergence (as $\lambda\to+\infty$) of ground state solutions for equation (\ref{eq-f}), when $b(x)=\lambda a(x)+1$ and $f(x,u)=\vert u\vert^{p-1}u$, where $a(x)\ge0$ satisfies $(B_1)$ and the potential well $\Omega=\{x\in V: a(x)=0\}$ is a non-empty connected and bounded domain in $V$. Similar results for $p$-Laplacian equations and biharmonic equations on locally finite graphs can be found in \cite{HSZ2020,HS2021}.

In this paper, we consider the following logarithmic Schr\"{o}dinger equation
\begin{equation}\label{eq-eq1}
-\Delta u+\lambda a(x)u=u\log u^2\ \ \ \ \mbox{in }V
\end{equation}
on a connected locally finite graph $G=(V,E)$, where the parameter $\lambda>0$. We recall that the logarithmic Schr\"odinger equation in Euclidean space
\begin{equation}\label{eq-log}
 -\Delta u+\lambda b(x)u=u\log u^2\ \ \mbox{in }\mathbb{R}^N
\end{equation}
has recently received much attention, one can see \cite{Cazenave1983,d'Avenia2014,d'Avenia2015,Guerrero2010,Ji2016,Shuai2019,Squassina2015,Tanaka2017,Wang2019} and references therein. Logarithmic nonlinear problems have a wide range of applications in fields such as quantum mechanics, quantum optics, nuclear physics, transport and diffusion phenomena, Bose-Einstein condensation and etc. (see \cite{Carles2018,Cazenave2003,Zloshchastiev2010,Y2020,WFHQ2023,PTGG2023}).

Different approaches have been developed to study the existence and multiplicity of solutions for nonlinear Schr\"odinger equations with logarithmic nonlinearities. Cazenave \cite{Cazenave1983} worked in an Orlicz space endowed with a Luxemburg type norm in order to make the associated energy functional of equation (\ref{eq-log}) to be $C^1$. Squassina and Szulkin \cite{Squassina2015} studied the existence of multiple solutions by using non-smooth critical point theory (see also \cite{d'Avenia2014,d'Avenia2015,Ji2016}). Tanaka and Zhang \cite{Tanaka2017} applied the penalization technique to study multi-bump solutions of equation (\ref{eq-log}). For the idea of penalization, see also \cite{AJ2022,AJ2020,Guerrero2010}. In \cite{Wang2019}, Wang and Zhang proved that the ground state solution of the power-law scalar field equations $-\Delta u+\lambda u=\vert u\vert^{p-2}u$, as $p\downarrow 2$, converge to the ground state solution of the logarithmic-law equation $-\Delta u=\lambda u\log u^2$. Recently, several results are devoted to studying the sign-changing solutions. Chen and Tang \cite{Chen2019} established the existence of least energy sign-changing solutions of some logarithmic Schr\"{o}dinger equation in bounded domains of $\mathbb{R}^N$. Shuai \cite{Shuai2019} obtained the existence of least energy sign-changing solutions for equation (\ref{eq-log}) under different types of potentials. Zhang and Wang investigated, in \cite{Zhang2020}, the existence and  concentration behaviors of sign-changing solutions for logarithmic scalar field equations in the semiclassical setting. Ji \cite{Ji2021} established the existence and multiplicity of multi-bump type nodal solutions for equation (\ref{eq-log}). For more studies on logarithmic nonlinear equations, one may refer to \cite{Ardila2016,Cazenave1983,Cazenave1982,d'Avenia2014,d'Avenia2015,Ji2016,Squassina2015} and their references.

The goal of this work is to show the existence of least energy sign-changing solutions of (\ref{eq-eq1}) and their asymptotic behavior as $\lambda\to+\infty$. To the best of our knowledge, there is no result on sign-changing solutions for logarithmic Schr\"{o}dinger problems on locally finite graphs.

In the sequel of this paper, we make the assumption that there exists a constant $\mu_{\min}>0$ such that the measure $\mu(x)\geq\mu_{\min}>0$ for all $x\in V$. As for the potential $a=a(x)$, we assume that:
\begin{itemize}
\item[$(A_1)$]~ $a(x)\geq0$ and the potential well $\Omega=\{x\in V:a(x)=0\}$ is a non-empty, connected and bounded domain in V;
\item[$(A_2)$]~ there exists $M>0$ such that the volume of the set $D_M$ is finite, namely,
    \[
    Vol(D_M)=\sum_{x\in D_M}\mu(x)<\infty,
    \]
    where $D_M=\{x\in V: a(x)<M\}$.
\end{itemize}

To explain our result, we first introduce some notations. For any function $u:V\to\mathbb{R}$, the graph Laplacian of $u$ is defined by
\begin{equation}\label{Laplacian}
  \Delta u(x)=\frac{1}{\mu(x)}\sum_{y\sim x}\omega_{xy}\left(u(y)-u(x)\right).
\end{equation}
The integral of $u$ over $V$ is defined by $\int_{V}ud\mu=\sum_{x\in V}\mu(x)u(x)$, and
the gradient form of the two functions $u, v$ on $V$ is defined by
\begin{equation}\label{gradient}
  \Gamma(u,v)(x)=\frac{1}{2\mu(x)}\sum_{y\sim x}\omega_{xy}\left(u(y)-u(x)\right)\left(v(y)-v(x)\right).
\end{equation}
Write $\Gamma(u)=\Gamma(u,u)$, and sometimes we use $\nabla u \nabla v$ to replace $\Gamma(u,v)$. The length of the gradient of $u$ is defined by
\begin{equation}\label{gradient length}
  \vert \nabla u\vert (x)=\sqrt{\Gamma(u)(x)}=\left(\frac{1}{2\mu(x)}\sum_{y\sim x}\omega_{xy}\left(u(y)-u(x)\right)^2\right)^{1/2}.
\end{equation}
Denote by $C_c(V)$ the set of all functions with compact support, and let $H^1(V)$ be the completion of $C_c(V)$ under the norm
\[
\|u\|_{H^1(V)}=\left( \int_V\left(\vert\nabla u\vert^2+u^2\right)d\mu \right)^{1/2}.
\]
Then, $H^1(V)$ is a Hilbert space with the inner product
\[
\langle u,v\rangle=\int_V\left(\Gamma(u,v)+uv\right)d\mu, \quad \ \forall u,\ v\in H^1(V).
\]
We write $\|u\|_p=\left(\int_V\vert u\vert^p d\mu \right)^{1/p}$ for $p\in[1,+\infty)$ and $\|u\|_{L^\infty}=\sup\limits_{x\in V}\vert u(x)\vert$.

For each $\lambda>0$ we introduce a space
\[
\mathcal{H}_\lambda=\left\{u\in H^1(V): \int_V\lambda a(x)u^2d\mu<+\infty\right\}
\]
with norm
\[
\|u\|_{\mathcal{H}_\lambda}^2\dot{=}\int_V\left(|\nabla u|^2+(\lambda a(x)+1)u^2\right)d\mu,
\]
which is induced by
\[
\langle u,v\rangle_{\mathcal{H}_\lambda}=\int_V\left(\Gamma(u,v)+(\lambda a(x)+1)uv\right)d\mu,\ \forall u,\ v\in \mathcal{H}_\lambda.
\]
Clearly, $\mathcal{H}_\lambda$ is also a Hilbert space.

Note that equation (\ref{eq-eq1}) is formally associated with the energy functional $J_\lambda:\ H^1(V)\to\mathbb{R}\cup\{+\infty\}$ given by
\begin{equation}\label{defJ}
J_\lambda(u)=\frac{1}{2}\int_V\left(\vert \nabla u\vert^2+(\lambda a(x)+1)u^2 \right)d\mu-\frac{1}{2}\int_Vu^2\log u^2d\mu.
\end{equation}
Clearly, $J_\lambda$ fails to be $C^1$ in $H^1(V)$. In fact, for some $G=(V,E)$ with suitable measure $\mu$, there exists $u\in H^1(V)$ but $\int_Vu^2\log u^2d\mu=-\infty$ (see \cite{CWY}).

In this paper, we consider the functional $J_\lambda$ in \eqref{defJ} on the set
\[
\mathcal{D}_\lambda=\left\{u\in \mathcal{H}_\lambda:\int_Vu^2\vert\log u^2\vert d\mu<\infty\right\}.
\]
That is,
\[
J_\lambda(u)=\frac{1}{2}\|u\|_{\mathcal{H}_\lambda}^2-\frac{1}{2}\int_Vu^2\log u^2d\mu, \quad \forall u\in \mathcal{D}_\lambda.
\]
Define the Nehari manifold and sign-changing Nehari set respectively by
\[
\mathcal{N}_\lambda=\left\{u\in\mathcal{D}_\lambda\setminus\{0\}: J'_\lambda(u)\cdot u=0\right\},
\]
\[
\mathcal{M}_\lambda=\left\{u\in\mathcal{D}_\lambda: u^\pm\neq0 \mbox{ and } J'_\lambda(u)\cdot u^+=J'_\lambda(u)\cdot u^-=0\right\},
\]
where $u^+=\max\{u,\ 0\}$ and $u^-=\min\{u,\ 0\}$. Clearly, $\mathcal{N}_\lambda$ contains all the nontrivial solutions of equation (\ref{eq-eq1}) and $\mathcal{M}_\lambda$ contains all the sign-changing solutions.
Set
\[
c_\lambda=\inf_{u\in\mathcal{N}_\lambda}J_\lambda(u),\ \ \
m_\lambda=\inf_{u\in\mathcal{M}_\lambda}J_\lambda(u).
\]

Our first result is as follows.

\begin{theorem}\label{th1.1}\quad
Suppose that $G=(V,E)$ is a connected locally finite graph and the potential $a: V\to\mathbb{R}$ satisfies $(A_1)$ and $(A_2)$. Then, there exists a constant $\lambda_0>0$ such that for all $\lambda\geq\lambda_0$, equation (\ref{eq-eq1}) admits a least energy sign-changing solution $u_{\lambda}\in\mathcal{D}_\lambda$ such that $J_\lambda(u_\lambda)=m_\lambda$. Moreover, $m_\lambda>2c_\lambda$.
\end{theorem}

We recall that $D\subset V$ is a bounded domain if the distance $d(x,y)$ between any $x,y\in D$ is uniformly bounded. The boundary of $D$ is defined by
\[
\partial D\dot{=}\{y\not\in D:\mbox{there exists}~ x\in D \mbox{ such that }xy\in E\}
\]
and the interior of $D$ is denoted by $D^{\circ}$. Obviously, $D^\circ = D$. Set $\Omega=\{x\in V: a(x)=0\}$.
Let $H_0^1(\Omega)$ be the completion of $C_c(\Omega)$ under the norm
\[
\|u\|_{H_0^1(\Omega)}=\left(\int_{\Omega\cup\partial\Omega}\vert\nabla u\vert^2d\mu+\int_{\Omega}u^2d\mu \right)^{1/2}.
\]
Then, $H_0^1(\Omega)$ is a Hilbert space.


We consider the following Dirichlet problem
\begin{equation}\label{eq-eq2}
\begin{cases}
-\Delta u=u\log u^2~~~&\mbox{ in }\Omega,\\
u(x)=0~~~&\mbox{ on }\partial\Omega.
\end{cases}
\end{equation}
The energy functional $J_{\Omega}:H_0^1(\Omega)\to\mathbb{R}$ associated with $(\ref{eq-eq2})$ is given by
\[
J_{\Omega}(u)\dot{=}\frac{1}{2}\|u\|_{H_0^1(\Omega)}^2-\frac{1}{2}\int_{\Omega}u^2\log u^2d\mu, \quad \forall u\in H_0^1(\Omega).
\]
Define
\[
\mathcal{N}_{\Omega}=\left\{u\in H_0^1(\Omega)\setminus\{0\}: J'_\Omega(u)\cdot u=0\right\},
\]
\[
\mathcal{M}_{\Omega}=\left\{u\in H_0^1(\Omega): u^\pm\neq0 \mbox{ and } J'_{\Omega}(u)\cdot u^+=J'_{\Omega}(u)\cdot u^-=0\right\}.
\]
Set
\[
c_\Omega=\inf_{u\in\mathcal{N}_\Omega}J_\lambda(u),\ \ \
m_{\Omega}=\inf_{u\in\mathcal{M}_{\Omega}}J_{\Omega}(u).
\]

Similar to Theorem \ref{th1.1}, problem $(\ref{eq-eq2})$ also has a least energy sign-changing solution.

\begin{theorem}\label{th1.2}\quad
Let $G=(V,E)$ be a connected locally finite graph. Assume $\Omega=\{x\in V:a(x)=0\}$ is a non-empty, connected and bounded domain in V. Then problem (\ref{eq-eq2}) admits a least energy sign-changing solution $u_0\in H_0^1(\Omega)$ such that $J_\Omega(u_\Omega)=m_\Omega$. Moreover, $m_\Omega>2c_\Omega$.
\end{theorem}

Finally, we show the convergence of $u_{\lambda}$ as $\lambda\to+\infty$.

\begin{theorem}\label{th1.3}\quad
Under the assumptions of Theorem \ref{th1.1}, we conclude that for any sequence $\lambda_k\to+\infty$, up to a subsequence, the corresponding least energy sign-changing solution $u_{\lambda_k}$ of equation $(\ref{eq-eq1})$ converges in $H^1(V)$ to a least energy sign-changing solution of problem $(\ref{eq-eq2})$.
\end{theorem}

One of the main challenges in proving Theorem \ref{th1.1}-\ref{th1.3} is to deal with the logarithmic term in equation (\ref{eq-eq1}). In Euclidean space, logarithmic Sobolev inequality is important in studying the logarithmic Schr\"odinger equation (see \cite{AJ2020,Shuai2019,Squassina2015} etc.). While, on discrete graphs, the logarithmic Sobolev inequality is only available under a positive curvature condition, which requires the measure $\mu$ to be finite (see \cite{Lin2017-3} for details). In our case, the measure $\mu$ has a uniform positive lower bound, which violates the positive curvature condition. To overcome this difficulty, we will develop new and delicate arguments not relying on the logarithmic Sobolev inequality.

In addition, in view of the associated energy functional with (\ref{eq-eq1}) is not well-defined, inspired by ideas in \cite{Shuai2019,Wang2019}, we will restrict $u^2\log u^2\in L^1(V)$. New challenge arises since the techniques in \cite{Shuai2019,Wang2019} are not applicable here because the graph Laplacian operator is non-local. To be precise, in \cite{Shuai2019}, denoting by $I$ the corresponding energy functional, we know that the following decomposition
\begin{eqnarray}\label{6-15-1}
I(u)=I(u^+)+I(u^-),  \quad \big<I'(u),u\big>=\big<I'(u^+),u^+\big>+\big<I'(u^-),u^-\big>,
\end{eqnarray}
plays a key role in studying nodal solutions. However, in our case, such a decomposition does not hold. Actually, by a direct computation, it follows that
for each $u\in\mathcal{D}_\lambda\setminus\{0\}$,
\begin{gather*}
J_\lambda(u)=J_\lambda(u^+)+J_\lambda(u^-)-\frac{1}{2}K_V(u),\\
J'_\lambda(u)\cdot u^{\pm}=J'_\lambda(u^{\pm})\cdot u^{\pm}-\frac{1}{2}K_V(u),
\end{gather*}
where $K_V(u)=\sum\limits_{x\in V}\sum\limits_{y\sim x}\omega_{xy}\left[u^+(x)u^-(y)+u^-(x)u^+(y)\right]<0$, see Section \ref{preliminary} for details. Clearly, $J(u)\neq J(u^+)+J(u^-)$ and $\big<J'(u),u\big>\neq \big<J'(u^+),u^+\big>+\big<J'(u^-),u^-\big>$, which imply that (\ref{6-15-1}) fails.  Motivated by \cite{CNW-2019,WZ-2016}, we will develop new variational arguments involving nonlocal operator based on  directional derivative to study (\ref{eq-eq1}).



The paper is organized as follows. In Section \ref{preliminary}, we introduce some notations, definitions and preliminary lemmas. In Section \ref{Proof of Theorem 1.1}, we apply the Nehari manifold method to prove the existence of least energy sign-changing solution of equation $(\ref{eq-eq1})$ and the Dirichlet problem $(\ref{eq-eq2})$. In Section \ref{Proof of Theorem 1.3}, we give the proof of Theorem \ref{th1.3}.

\section{Some preliminary results}\label{preliminary}
\subsection{Some definitions}

To prove Theorem \ref{th1.1}, we introduce the following definitions.
\begin{definition}\label{defdirect}\quad
Given $u\in\mathcal{D}_\lambda$ and $\phi\in C_c(V)$, the derivative of $J_\lambda$ in the direction $\phi$ at $u$, denoted by $J'_\lambda(u)\cdot\phi$, is defined as $\lim_{t\to0^+}\frac{1}{t}\left[J_\lambda(u+t\phi)-J_\lambda(u)\right]$.
\end{definition}

\begin{definition}\label{def1}\quad
\begin{itemize}
\item[(1)] For $u,v\in\mathcal{D}_\lambda$, we define
\[
J_\lambda'(u)\cdot v:=\int_V\left(\Gamma(u,v)+\lambda a(x)uv\right)d\mu-\int_Vuv\log u^2d\mu.
\]
Clearly, $\int_Vuv\log u^2d\mu$ is well-defined for $u,v\in\mathcal{D}_\lambda$.
\item[(2)] We say that $u\in\mathcal{H}_\lambda$ is a critical point of $J_\lambda$ if $u\in\mathcal{D}_\lambda$ and $J_\lambda'(u)\cdot v=0$ for all $v\in\mathcal{D}_\lambda$. We also say that $d_\lambda\in\mathbb{R}$ is a critical value for $J_\lambda$ if there exists a critical point $u\in\mathcal{H}_\lambda$ such that $J_\lambda(u)=d_\lambda$.
\end{itemize}
\end{definition}
It is easily seen that, $u$ is a weak solution to equation (\ref{eq-eq1}) if and only if $u$ is a critical point of $J_\lambda$.

Note that, for any $0<\varepsilon<1$, there exists $C_{\varepsilon}>0$ such that
\[
|u^2\log u^2|\leq C_{\varepsilon}(|u|^{2-\varepsilon}+|u|^{2+\varepsilon}).
\]
Since $H^1(\Omega)\hookrightarrow L^p(\Omega)$ is compact for $p\in[1,+\infty]$, we have $J_{\Omega}\in C^1(H_0^1(\Omega),\mathbb{R})$ and
\[
J'_{\Omega}(u)\cdot v=\int_{\Omega\cup\partial\Omega}\nabla u\nabla vd\mu-\int_{\Omega}uv\log u^2d\mu, \forall u,v\in H_0^1(\Omega).
\]
Clearly, $u$ is a weak solution to problem (\ref{eq-eq2}) if and only if $u$ is a critical point of $J_{\Omega}$.

\begin{lemma}\label{lemp-wsolu}\quad
If $u\in\mathcal{D}_\lambda$ is a weak solution of (\ref{eq-eq1}), then $u$ is a point-wise solution of (\ref{eq-eq1}).
\end{lemma}
\begin{proof}\quad
If $u\in\mathcal{D}_\lambda$ is a weak solution of (\ref{eq-eq1}), then for any $\varphi\in\mathcal{D}_\lambda$, there holds
\[
\int_V\left(\Gamma(u,\varphi)+\lambda a(x)u\varphi\right)d\mu=\int_Vu\varphi\log u^2d\mu.
\]
Using $C_c(V)$ is dense in $\mathcal{D}_\lambda$ and $\omega_{x,y}$ is symmetric, for any $\varphi\in C_c(V)$, by integration by parts, we have
\[\begin{aligned}
\int_V\Gamma(u,\varphi)d\mu=&\frac{1}{2}\sum_{x\in V}\sum_{y\sim x}\omega_{xy}\left(u(y)-u(x)\right)\left(\varphi(y)-\varphi(x)\right)\\
=&-\frac{1}{2}\sum_{y\in V}\sum_{x\sim y}\omega_{xy}\left(u(y)-u(x)\right)\varphi(x)-\frac{1}{2}\sum_{x\in V}\sum_{y\sim x}\omega_{xy}\left(u(y)-u(x)\right)\varphi(x)\\
=&-\sum_{x\in V}\sum_{y\sim x}\omega_{xy}\left(u(y)-u(x)\right)\varphi(x)\\
=&-\int_V\Delta u\varphi d\mu,
\end{aligned}\]
which gives
\begin{equation}\label{eqpoint-wise}
\int_V\left(-\Delta u+\lambda a(x)u\right)\varphi d\mu=\int_Vu\varphi\log u^2d\mu,\ \ \forall \varphi\in C_c(V).
\end{equation}
For any fixed $y\in V$, take a test function $\varphi: V\to\mathbb{R}$ in (\ref{eqpoint-wise}) with
\[
\varphi(x)=
\begin{cases}
1, ~ &x=y,\\
0, ~ &x\neq y.
\end{cases}
\]
Clearly, $\varphi\in\mathcal{D}_\lambda$ and $-\Delta u(y)+\lambda a(y)u(y)-u(y)\log \left(u(y)\right)^2=0$.
Since $y$ is arbitrary, we conclude that $u$ is a point-wise solution of (\ref{eq-eq1}).
\end{proof}

Similarly, we obtain
\begin{lemma}\quad
If $u\in H_0^1(\Omega)$ is a weak solution of (\ref{eq-eq2}), then $u$ is a point-wise solution of (\ref{eq-eq2}).
\end{lemma}

Next, we have the following observations:
\begin{eqnarray}
&&\int_V\Gamma(u^++u^-)d\mu =\int_V\Gamma(u^+)d\mu+\int_V\Gamma(u^-)d\mu-K_V(u),\label{eqGammau+u-}\\
&&\int_V\Gamma(u^++u^-,u^+)d\mu=\int_V\Gamma(u^+)d\mu-\frac{1}{2}K_V(u),\label{eqGammau+}\\
&&\int_V\Gamma(u^++u^-,u^-)d\mu=\int_V\Gamma(u^-)d\mu-\frac{1}{2}K_V(u).\label{eqGammau-}
\end{eqnarray}
Then, for each $u\in\mathcal{D}_\lambda$, we have
\begin{gather*}
J_\lambda(u)=J_\lambda(u^+)+J_\lambda(u^-)-\frac{1}{2}K_V(u),\\
J'_\lambda(u)\cdot u^{\pm}=J'_\lambda(u^{\pm})\cdot u^{\pm}-\frac{1}{2}K_V(u),
\end{gather*}
and for each $u\in H_0^1(\Omega)$,
\begin{gather*}
J_{\Omega}(u)=J_{\Omega}(u^+)+J_{\Omega}(u^-)-\frac{1}{2}K_\Omega(u),\\
J'_{\Omega}(u)\cdot u^{\pm}=J'_{\Omega}(u^{\pm})\cdot u^{\pm}-\frac{1}{2}K_\Omega(u),
\end{gather*}
where $K_\Omega(u):=\sum\limits_{x\in\Omega\cup\partial\Omega}\sum\limits_{y\sim x}\omega_{xy}\left[u^+(x)u^-(y)+u^-(x)u^+(y)\right]$.

\subsection{Sobolev embedding}

In this subsection, we establish a Sobolev embedding result.

\begin{lemma}\label{lemA2}\quad
If $\mu(x)\geq\mu_{\min}>0$ and $a(x)$ satisfies $(A_1)-(A_2)$, then there exist a constant $\lambda_0>0$ such that, for all $\lambda\geq\lambda_0$, the space $\mathcal{H}_\lambda$ is compactly embedded into $L^p(V)$ for all $2\leq p\leq+\infty$.
\end{lemma}
\begin{proof}\quad
For all $\lambda>0$, at any vertex $x_0\in V$, by $(A_1)$ we have
\[\begin{aligned}
\|u\|_{\mathcal{H}_\lambda}^2=&\int_V\left(|\nabla u|^2+(\lambda a(x)+1)u^2\right)d\mu\geq\int_Vu^2d\mu\geq\mu_{\min}u^2(x_0),
\end{aligned}\]
which implies that $\vert u(x_0)\vert\leq\sqrt{\frac{1}{\mu_{\min}}}\|u\|_{\mathcal{H}_\lambda}$. Thus $\mathcal{H}_\lambda\hookrightarrow L^\infty(V)$ continuously. Hence, using interpolation gives that $\mathcal{H}_\lambda\hookrightarrow L^p(V)$ continuously for all $2\leq p\leq\infty$. Assuming $\{u_k\}$ is bounded in $\mathcal{H}_\lambda$, we have that, up to a subsequence, $u_k\rightharpoonup u$ in $\mathcal{H}_\lambda$. In particular, $\{u_k\}\subset\mathcal{H}_\lambda$ is also bounded in $L^2(V)$ and by the weak convergence in $L^2(V)$ it follows that, for any $\varphi\in L^2(V)$,
\begin{equation}\label{point-wisely}
\lim_{k\to\infty}\int_V(u_k-u)\varphi d\mu=\lim_{k\to\infty}\sum_{x\in V}\mu(x)\left(u_k(x)-u(x)\right)\varphi(x)=0.
\end{equation}
Take any $x_0\in V$ and let
\[
\varphi_0(x)=
\begin{cases}
1, ~ x=x_0,\\
0, ~ x\neq x_0.
\end{cases}
\]
Obviously, $\varphi_0(x)\in L^2(V)$. By substituting $\varphi_0$ into (\ref{point-wisely}), we get $\lim_{k\to\infty}u_k(x)=u(x)$ for any fixed $x\in V$.

Since $u_k$ is bounded in $\mathcal{H}_\lambda$ and $u\in\mathcal{H}_\lambda$, there exists $C_1>0$ such that
\[
\lambda\int_Va(x)(u_k-u)^2d\mu\leq C_1.
\]
We claim that, up to a subsequence,
\[
\lim_{k\to+\infty}\int_V(u_k-u)^2d\mu=0.
\]
In fact, since $a(x)$ satisfies $(A_2)$, there exists some $M>0$ such that
\[\begin{aligned}
\int_V(u_k-u)^2d\mu=&\int_{D_M}(u_k-u)^2d\mu+\int_{V\setminus D_M}(u_k-u)^2d\mu\\
\leq&\int_{D_M}(u_k-u)^2d\mu+\int_{V\setminus D_M}\frac{1}{\lambda M}\lambda a(x)(u_k-u)^2d\mu\\
\leq&\int_{D_M}(u_k-u)^2d\mu+\frac{C_1}{\lambda M}.
\end{aligned}\]
Then, for all $\varepsilon>0$, there exists $\lambda_0>0$ such that when $\lambda>\lambda_0$, we have $\frac{C_1}{\lambda M}<\varepsilon$. Moreover, up to a subsequence, we have
\[
\lim_{k\to+\infty}\int_{D_M}(u_k-u)^2d\mu=0.
\]
Hence the claim holds. In view of $\|u_k-u\|_{\infty}^2\leq\frac{1}{\mu_{\min}}\int_V\vert u_k-u\vert^2d\mu$,
for any $2<p<\infty$, we deduce
\[
\int_V\vert u_k-u\vert^pd\mu\leq\left(\frac{1}{\mu_{\min}}\right)^{\frac{p-2}{2}}\left(\int_V\vert u_k-u\vert^2d\mu\right)^{\frac{p}{2}}.
\]
Therefore, up to a subsequence, $u_k\to u$ in $L^p(V)$ for all $2\leq p\leq+\infty$.
\end{proof}

\section{Existence of least energy sign-changing solutions}\label{Proof of Theorem 1.1}

This section is devoted to proving that equation (\ref{eq-eq1}), as well as $(\ref{eq-eq2})$, admits a least energy sign-changing solution.

The following result will be useful.

\begin{lemma}\label{lemJJst}\quad
For all $u\in\mathcal{M}_\lambda$ and $s, t>0$, there holds
\[
J_\lambda(u)\geq J_\lambda(su^++tu^-).
\]
The "=" holds if and only if $s=t=1$.
\end{lemma}
\begin{proof}\quad
For any $u\in\mathcal{M}_\lambda$,
\[\begin{aligned}
J_\lambda(u)=&J_\lambda(u)-\frac{1}{2}J'_\lambda(u)\cdot u^+-\frac{1}{2}J'_\lambda(u)\cdot u^-\\
=&\left(\frac{1}{2}\|u^+\|_{\mathcal{H}_\lambda}^2-\frac{1}{2}\int_V\vert u^+\vert^2\log\vert u^+\vert^2d\mu\right)-\left(\frac{1}{2}\|u^+\|_{\mathcal{H}_\lambda}^2-\frac{1}{2}\int_V\vert u^+\vert^2\log\vert u^+\vert^2d\mu-\frac{1}{2}\|u^+\|_2^2\right)\\
&+\left(\frac{1}{2}\|u^-\|_{\mathcal{H}_\lambda}^2-\frac{1}{2}\int_V\vert u^-\vert^2\log\vert u^-\vert^2d\mu\right)-\left(\frac{1}{2}\|u^-\|_{\mathcal{H}_\lambda}^2-\frac{1}{2}\int_V\vert u^-\vert^2\log\vert u^-\vert^2d\mu-\frac{1}{2}\|u^+\|_2^2\right)\\
=&\frac{1}{2}\|u^+\|_2^2+\frac{1}{2}\|u^-\|_2^2.
\end{aligned}\]
For $s,t>0$, by (\ref{eqGammau+u-}) we obtain
\begin{equation*}\label{eqGammasu+tu-}
\begin{aligned}
\int_V\Gamma(su^++tu^-)d\mu
=\int_V\Gamma(su^+)d\mu+\int_V\Gamma(tu^-)-stK_V(u).
\end{aligned}
\end{equation*}
Hence,
\[\begin{aligned}
&J_\lambda(su^++tu^-)\\
=&J_\lambda(su^+)+J_\lambda(tu^-)-\frac{st}{2}K_V(u)\\
=&s^2J_\lambda(u^+)-\frac{1}{2}s^2\log s^2\|u^+\|_2^2+t^2J_\lambda(u^-)-\frac{1}{2}t^2\log t^2\|u^-\|_2^2-\frac{st}{2}K_V(u)\\
=&s^2\left[J_\lambda(u^+)-\frac{1}{2}J'_\lambda(u)\cdot u^+\right]-\frac{1}{2}s^2\log s^2\|u^+\|_2^2+t^2\left[J_\lambda(u^-)-\frac{1}{2}J'_\lambda(u)\cdot u^-\right]\\
&-\frac{1}{2}t^2\log t^2\|u^-\|_2^2-\frac{st}{2}K_V(u)\\
=&s^2\left[J_\lambda(u^+)-\frac{1}{2}J'_\lambda(u^+)\cdot u^++\frac{1}{4}K_V(u)\right]-\frac{1}{2}s^2\log s^2\|u^+\|_2^2\\
&+t^2\left[J_\lambda(u^-)-\frac{1}{2}J'_\lambda(u^-)\cdot u^-+\frac{1}{4}K_V(u)\right]-\frac{1}{2}t^2\log t^2\|u^-\|_2^2-\frac{st}{2}K_V(u)\\
=&\frac{1}{2}(s^2-s^2\log s^2)\|u^+\|_2^2+\frac{1}{2}(t^2-t^2\log t^2)\|u^-\|_2^2+\frac{(s-t)^2}{4}K_V(u).
\end{aligned}\]
Therefore, defining $f(\tau)=\tau^2-\tau^2\log \tau^2-1$ for any $\tau\ge0$, we have
\[\begin{aligned}
&J_\lambda(su^++tu^-)-J_\lambda(u)\\
=&\frac{1}{2}(s^2-s^2\log s^2-1)\|u^+\|_2^2+\frac{1}{2}(t^2-t^2\log t^2-1)\|u^-\|_2^2+\frac{(s-t)^2}{4}K_V(u)\\
=&\frac{1}{2}f(s)\|u^+\|_2^2+\frac{1}{2}f(t)\|u^-\|_2^2+\frac{(s-t)^2}{4}K_V(u).
\end{aligned}\]
Since $f(0)=-1$, $f(1)=0$ and $f(\tau)<0$ if $\tau\neq1$, $\frac{(s-t)^2}{4}K_V(u)<0$ for any $s\neq t$, the conclusions follow.\end{proof}

Next we show $\mathcal{M}_\lambda\neq\emptyset$.

\begin{lemma}\label{lememptyset}\quad
If $u\in \mathcal{D}_\lambda\setminus\{0\}$ with $u^{\pm}\neq0$, then there exists a unique positive number pair $(s_u,t_u)$ satisfying $s_uu^++t_uu^-\in\mathcal{M}_\lambda$.
\end{lemma}

\begin{proof}\quad
For $s,t>0$, by (\ref{eqGammau+}) and (\ref{eqGammau-}), we get
\begin{equation*}\label{eqGammasu+}
\begin{aligned}
\int_V\Gamma(su^++tu^-,su^+)d\mu=\int_V\Gamma(su^+)d\mu-\frac{st}{2}K_V(u)
\end{aligned}
\end{equation*}
and
\begin{equation*}\label{eqGammatu-}
\begin{aligned}
\int_V\Gamma(su^++tu^-,tu^-)d\mu=\int_V\Gamma(tu^-)d\mu-\frac{st}{2}K_V(u).
\end{aligned}
\end{equation*}
Let
\[\begin{aligned}
g_1(s,t)\doteq&J'_\lambda(su^++tu^-)\cdot(su^+)\\
=&s^2\|u^+\|_{\mathcal{H}_\lambda}^2-s^2\int_V\vert u^+\vert^2\log\vert u^+\vert^2d\mu-s^2\log s^2\|u^+\|_2^2-s^2\|u^+\|_2^2-\frac{st}{2}K_V(u)
\end{aligned}\]
and
\[\begin{aligned}
g_2(s,t)\doteq&J'_\lambda(su^++tu^-)\cdot(tu^-)\\
=&t^2\|u^-\|_{\mathcal{H}_\lambda}^2-t^2\int_V\vert u^-\vert^2\log\vert u^-\vert^2d\mu-t^2\log t^2\|u^-\|_2^2-t^2\|u^-\|_2^2-\frac{st}{2}K_V(u).
\end{aligned}\]
We can see that there exists $r_1>0$ small enough and $R_1>0$ large enough such that
\[\begin{aligned}
&g_1(s,s)>0,\ g_2(s,s)>0\ \mbox{for all }s\in(0,r_1),\\
&g_1(s,s)<0,\ g_2(s,s)<0\ \mbox{for all }s\in(R_1,+\infty).
\end{aligned}\]
Hence, there exist $0<r<R$ such that
\[\begin{aligned}
&g_1(r,t)>0,\ g_1(R,t)<0\ \mbox{for all }t\in[r,R],\\
&g_2(s,r)>0,\ g_2(s,R)<0\ \mbox{for all }s\in[r,R].
\end{aligned}\]
Applying Miranda's theorem \cite{M}, there exist some $s_u,t_u\in[r,R]$ such that $g_1(s_u,t_u)=g_2(s_u,t_u)=0$, which implies that $s_uu^++t_uu^-\in\mathcal{M}_\lambda$.

In what follows, we prove the uniqueness of $(s_u,t_u)$. If $u\in\mathcal{M}_\lambda$, then
\begin{equation}\label{eqJ'u+}
0=J'_\lambda(u)\cdot u^+=J'_\lambda(u^+)\cdot u^+-\frac{1}{2}K_V(u)
\end{equation}
and
\begin{equation}\label{eqJ'u-}
0=J'_\lambda(u)\cdot u^-=J'_\lambda(u^-)\cdot u^--\frac{1}{2}K_V(u).
\end{equation}
We claim that $(s_u, t_u)=(1,1)$ is the unique pair such that $s_uu^++t_uu^-\in\mathcal{M}_\lambda$. Indeed, if $(s_u, t_u)=(1,1)$ satisfies $s_uu^++t_uu^-\in\mathcal{M}_\lambda$, without loss of generality, we assume that $0<s_u\leq t_u$. Then
\begin{equation}\label{eqJ'suu+}
\begin{aligned}
0=&J'_\lambda(s_uu^++t_uu^-)\cdot (s_uu^+)\\
=&s_u^2J'_\lambda(u^+)\cdot u^+-s_u^2\log s_u^2\|u^+\|_2^2-\frac{s_ut_u}{2}K_V(u)\\
\geq&s_u^2J'_\lambda(u^+)\cdot u^+-s_u^2\log s_u^2\|u^+\|_2^2-\frac{s_u^2}{2}K_V(u)
\end{aligned}
\end{equation}
and
\begin{equation}\label{eqJ'tuu-}
\begin{aligned}
0=&J'_\lambda(s_uu^++t_uu^-)\cdot (t_uu^-)\\
=&t_u^2J'_\lambda(u^-)\cdot u^--t_u^2\log t_u^2\|u^-\|_2^2-\frac{s_ut_u}{2}K_V(u)\\
\leq&t_u^2J'_\lambda(u^-)\cdot u^--t_u^2\log t_u^2\|u^-\|_2^2-\frac{t_u^2}{2}K_V(u).
\end{aligned}
\end{equation}
Together with \eqref{eqJ'u+} and \eqref{eqJ'suu+}, we get
\[
s_u^2\log s_u^2\int_V|u^+|^2d\mu\geq0,
\]
Similarly, by \eqref{eqJ'u-} and \eqref{eqJ'tuu-}, we can deduce that
\[
t_u^2\log t_u^2\int_V|u^-|^2d\mu\leq0,
\]
which implies that $s_u\geq1$ and $t_u\leq1$. In view of $0<s_u\leq t_u$, it follows that $s_u=t_u=1$.

If $u\not\in\mathcal{M}_\lambda$, let $(s_1,t_1)$ and $(s_2,t_2)$ be two different positive pairs such that $v_i:=s_iu^++t_iu^-\in\mathcal{M}_\lambda,i=1,2$, which shows that
\[
\frac{s_2}{s_1}v_1^++\frac{t_2}{t_1}v_1^-=v_2\in\mathcal{M}_\lambda.
\]
 By similar analysis as above, we obtain $\frac{s_2}{s_1}=\frac{s_2}{s_1}=1$.
This implies that $(s_1,t_1)=(s_2,t_2)$.
\end{proof}

\begin{lemma}\label{lemstleq1}\quad
Let $u\in\mathcal{D}_\lambda$ with $u^\pm\neq0$ such that $J'_\lambda(u)\cdot u^\pm\leq0$. Then the unique pair $(s_u,t_u)$ obtained in Lemma \ref{lememptyset} satisfies $s_u,t_u\in(0,1]$. In particular, the "=" holds if and only if $s_u=t_u=1$.
\end{lemma}

\begin{proof}\quad
Without loss of generality, we assume that $0<t_u\leq s_u$. Since $s_uu^++t_uu^-\in\mathcal{M}_\lambda$, we have
\begin{eqnarray}\label{eqJ'suu+'}
0=J'_\lambda(s_uu^++t_uu^-)\cdot(s_uu^+)=s_u^2J'_\lambda(u^+)\cdot u^+-s_u^2\log s_u^2\|u^+\|_2^2-\frac{s_ut_u}{2}K_V(u).
\end{eqnarray}
Note that $K_V(x,y)<0$, using $J'_\lambda(u)\cdot u^+\leq0$ and \eqref{eqJ'suu+'}, we can deduce that
\[\begin{aligned}
0\leq&s_u^2\left(J'_\lambda(u^+)\cdot u^+-\frac{1}{2}K_V(x,y)\right)-s_u^2\log s_u^2\|u^+\|_2^2\\
=&s_u^2J'_\lambda(u)\cdot u^+-s_u^2\log s_u^2\|u^+\|_2^2\\
\leq&-s_u^2\log s_u^2\|u^+\|_2^2,
\end{aligned}\]
which implies that $0<s_u\leq1$. Therefore, $0<t_u\leq s_u\leq1$.
\end{proof}
Similarly, we have
\begin{lemma}\label{coremptyset}\quad
If $u\in H_0^1({\Omega})\setminus\{0\}$ with $u^{\pm}\neq0$, then there exists a unique positive number pair $(s_u,t_u)$ satisfying $s_uu^++t_uu^-\in \mathcal{M}_\Omega$.
\end{lemma}
\begin{lemma}\label{corstleq1}\quad
Let $u\in H_0^1(\Omega)$ with $u^\pm\neq0$ such that $J'_\Omega(u)\cdot u^\pm\leq0$. Then the unique pair $(s_u,t_u)$ obtained in Lemma \ref{coremptyset} satisfies $s_u,t_u\in(0,1]$. In particular, the "=" holds if and only if $s_u=t_u=1$.
\end{lemma}

Now we prove that the minimizer of $J_{\lambda}$ on $\mathcal{M}_{\lambda}$ is achieved.
\begin{lemma}\label{lemm-lambda>0}\quad
Supposed $(A_1)$ and $(A_2)$ hold. Then $m_\lambda>0$ is achieved.
\end{lemma}
\begin{proof}\quad
Taking a minimizing sequence $\{u_k\}\subset\mathcal{M}_\lambda$ of $J_\lambda$ yields
\begin{equation}\label{eqm-lambda}
\begin{aligned}
\lim_{k\to+\infty}J_\lambda(u_k)=&\lim_{k\to+\infty}\left[J_\lambda(u_k)-\frac{1}{2}J'_\lambda(u_k)\cdot u_k^+-\frac{1}{2}J'_\lambda(u_k)\cdot u_k^-\right]\\
=&\lim_{k\to+\infty}\left[J_\lambda(u_k^+)-\frac{1}{2}J'_\lambda(u_k^+)\cdot u_k^++J_\lambda(u_k^-)-J'_\lambda(u_k^-)\cdot u_k^-\right]\\
=&\lim_{k\to+\infty}\left(\frac{1}{2}\|u_k^+\|_2^2+\frac{1}{2}\|u_k^-\|_2^2\right)
=m_\lambda.
\end{aligned}
\end{equation}
By Lemma \ref{lemA2}, the H\"{o}lder's inequality and Young inequality, for any $\varepsilon\in(0,1)$, there exist $C_\varepsilon, C_\varepsilon', C_\varepsilon''>0$ such that
\[\begin{aligned}
\int_V|u_k^\pm|^2\log \vert u_k^\pm\vert^2d\mu\leq&\int_V(\vert u_k^\pm\vert^2\log \vert u_k^\pm\vert^2)^+d\mu\leq C_\varepsilon\int_V|u_k^\pm|^{2+\varepsilon}d\mu\\
\leq&C_\varepsilon\left(\int_V\vert u_k^\pm\vert^2d\mu\right)^{\frac{1}{2}}\left(\int_V\vert u_k^\pm\vert^{2(1+\varepsilon)}d\mu\right)^{\frac{1}{2}}\\
\leq&C_\varepsilon'\|u_k^\pm\|_2\|u_k^\pm\|_{\mathcal{H}_\lambda}^{1+\varepsilon}\\
\leq&\frac{1}{2}\|u_k^\pm\|_{\mathcal{H}_\lambda}^2+C_\varepsilon''\|u_k^\pm\|_2^{\frac{2}{1-\varepsilon}}.
\end{aligned}\]
Since $\{u_k\}\subset\mathcal{M}_\lambda$, we deduce that
\begin{equation}\label{eq2}
\begin{aligned}
\|u_k^\pm\|_{\mathcal{H}_\lambda}^2-\frac{1}{2}K_V^k(x,y)
\leq\frac{1}{2}\|u_k^\pm\|_{\mathcal{H}_\lambda}^2+C_\varepsilon''\|u_k^\pm\|_2^{\frac{2}{1-\varepsilon}}+\|u_k^\pm\|_2^2,
\end{aligned}
\end{equation}
where $K_V^k(u)=\sum\limits_{x\in V}\sum\limits_{y\sim x}\left[u_k^+(x)u_k^-(y)+u_k^-(x)u_k^+(y)\right]$.
This together with (\ref{eqm-lambda}) implies that $\{u_k^\pm\}$ is bounded in $\mathcal{H}_\lambda$ and $\{u_k\}$ is also bounded in $\mathcal{H}_\lambda$. Then, there exists $\lambda_0>0$ such that $\lambda\geq\lambda_0$, by Lemma \ref{lemA2}, there exists $u_{\lambda}\in\mathcal{H}_\lambda$ such that
\[\begin{cases}
u_k\rightharpoonup u_{\lambda} ~ \mbox{ weakly in }\mathcal{H}_\lambda,\\
u_k\to u_{\lambda} ~ \mbox{ point-wisely in } V,\\
u_k\to u_{\lambda} ~ \mbox{ strongly in } L^p(V) \mbox{ for } p\in[2,+\infty].
\end{cases}\]
Thus, together with the weak-lower semi-continuity of norm and Fatou's lemma, we get
\begin{gather*}
\begin{aligned}
&\int_V\left(\Gamma(u_{\lambda}^+)+\left(\lambda a(x)+1\right)\vert u_{\lambda}^+\vert^2\right)d\mu-\int_V(\vert u_{\lambda}^+\vert^2\log\vert u_{\lambda}^+\vert^2)^-d\mu-\frac{1}{2}K_V^\lambda(u)\\
\leq&\liminf_{k\to+\infty}\left[\int_V\left(\Gamma(u_k^+)+\left(\lambda a(x)+1\right)\vert u_k^+\vert^2\right)d\mu-\int_V(\vert u_k^+\vert^2\log\vert u_k^+\vert^2)^-d\mu-\frac{1}{2}K_V^k(u)\right]\\
=&\liminf_{k\to+\infty}\int_V\left(\vert u_k^+\vert^2+(\vert u_k^+\vert^2\log\vert u_k^+\vert^2)^+\right)d\mu\\
=&\int_V\vert u_{\lambda}^+\vert^2d\mu+\int_V(\vert u_{\lambda}^+\vert^2\log\vert u_{\lambda}^+\vert^2)^+d\mu,
\end{aligned}
\end{gather*}
where $K_V^\lambda(u)=\sum\limits_{x\in V}\sum\limits_{y\sim x}\left[u_\lambda^+(x)u_\lambda^-(y)+u_\lambda^-(x)u_\lambda^+(y)\right]$.
It follows that
\begin{equation*}\label{equ0+}
J'_\lambda(u_\lambda)\cdot u_\lambda^+=\int_V\left(\Gamma(u_{\lambda}^+)+\lambda a(x)\vert u_{\lambda}^+\vert^2\right)d\mu-\int_V\vert u_{\lambda}^+\vert^2\log\vert u_{\lambda}^+\vert^2d\mu-\frac{1}{2}K_V^\lambda(u)\leq0.
\end{equation*}
Similarly, it holds that
\begin{equation*}\label{equ0-}
J'_\lambda(u_\lambda)\cdot u_\lambda^-=\int_V\left(\Gamma(u_{\lambda}^-)+\lambda a(x)\vert u_{\lambda}^-\vert^2\right)d\mu-\int_V\vert u_{\lambda}^-\vert^2\log\vert u_{\lambda}^-\vert^2d\mu-\frac{1}{2}K_V^\lambda(u)\leq0.
\end{equation*}
In view of Lemma \ref{lememptyset} and Lemma \ref{lemstleq1}, there exist $s, t\in(0,1]$ such that $\widetilde{u}=su_{\lambda}^++tu_{\lambda}^-\in\mathcal{M}_\lambda$. Then
\begin{gather*}
\begin{aligned}
m_\lambda\leq&J_\lambda(\widetilde{u})=J_\lambda(\widetilde{u})-\frac{1}{2}J'_\lambda(\widetilde{u})\cdot(su_{\lambda}^+)-\frac{1}{2}J'_\lambda(\widetilde{u})\cdot(tu_{\lambda}^-)\\
=&\frac{s^2}{2}\|u_{\lambda}^+\|_2^2+\frac{t^2}{2}\|u_{\lambda}^-\|_2^2\\
\leq&\liminf_{k\to+\infty}\left[\frac{1}{2}\|u_k^+\|_2^2+\frac{1}{2}\|u_k^-\|_2^2\right]\\
=&\liminf_{k\to+\infty}\left[J_\lambda(u_k)-\frac{1}{2}J'_\lambda(u_k)\cdot u_k^+-\frac{1}{2}J'_\lambda(u_k)\cdot u_k^-\right]\\
=&\liminf_{k\to+\infty}J_\lambda(u_k)=m_\lambda.
\end{aligned}
\end{gather*}
This implies that $s=t=1$, i.e., $u_\lambda\in\mathcal{M}_\lambda$ satisfying $J_\lambda(u_{\lambda})=m_\lambda$.

We claim that $m_\lambda>0$. In fact, if $m_\lambda=0$, we have
\[
0=J_\lambda(u_{\lambda})-\frac{1}{2}J'(u_{\lambda})\cdot u_{\lambda}^+-\frac{1}{2}J'(u_{\lambda})\cdot u_{\lambda}^-=\frac{1}{2}\|u_{\lambda}^+\|_2^2+\frac{1}{2}\|u_\lambda^-\|_2^2.
\]
Then, by similar arguments as in (\ref{eq2}), it follows that $\|u_{\lambda}^\pm\|_{\mathcal{H}_\lambda}=0$. However, by Lemma \ref{lemA2}, for any $q>2$, there exists $C_q>0$ such that
\[
\|u_{\lambda}^\pm\|_{\mathcal{H}_\lambda}^2<\int_V\vert u_{\lambda}^\pm\vert^2\log \vert u_{\lambda}^\pm\vert^2d\mu\leq\int_V(\vert u_{\lambda}^\pm\vert^2\log \vert u_{\lambda}^\pm\vert^2)^+d\mu\leq C_q\int_V\vert u_{\lambda}^\pm\vert^qd\mu\leq C\|u_{\lambda}^\pm\|_{\mathcal{H}_\lambda}^q,
\]
which provides a contradiction. Hence the claim holds.
\end{proof}

The following lemma completes the proof of Theorem \ref{th1.1}.
\begin{lemma}\label{lemsign-solu}\quad
If $u\in\mathcal{M}_\lambda$ with $J_\lambda(u)=m_\lambda$, then $u$ is a sign-changing solution of equation (\ref{eq-eq1}). Moreover, $m_\lambda>2c_\lambda$.
\end{lemma}

\begin{proof}\quad
We assume by contradiction that $u\in\mathcal{M}_\lambda$ with $J_\lambda(u)=m_\lambda$, but $u$ is not a solution of equation (\ref{eq-eq1}). Then we can find a function $\phi\in C_c(V)$ such that
\[
\int_V\left(\nabla u\nabla\phi+\lambda a(x)u\phi\right)d\mu-\int_Vu\phi\log u^2d\mu\leq-1,
\]
which implies that, for some $\varepsilon>0$ small enough,
\[
J_\lambda'(su^++tu^-+\sigma\phi)\cdot\phi\leq-\frac{1}{2}\ \mbox{for all }\vert s-1\vert+\vert t-1\vert+\vert\sigma\vert\leq\epsilon.
\]
In what follows, we estimate $\sup\limits_{s,t}J_\lambda\left(su^++tu^-+\varepsilon\eta(s,t)\phi\right)$, where $\eta$ is a cut-off function such that
\[\eta(s,t)=
\begin{cases}
1 ~~~ \mbox{if }\vert s-1\vert\leq\frac{1}{2}\varepsilon \mbox{ and } \vert t-1\vert\leq\frac{1}{2}\varepsilon,\\
0 ~~~ \mbox{if }\vert s-1\vert\geq\varepsilon \mbox{ or } \vert t-1\vert\geq\varepsilon.
\end{cases}\]
In the case of $\vert s-1\vert\leq\varepsilon$ and $\vert t-1\vert\leq\varepsilon$, we have
\[\begin{aligned}
J_\lambda\left(su^++tu^-+\varepsilon\eta(s,t)\phi\right)=&J_\lambda\left(su^++tu^-+\varepsilon\eta(s,t)\phi\right)-J_\lambda(su^++tu^-)+J_\lambda(su^++tu^-)\\
=&J_\lambda(su^++tu^-)+\int_0^1 J_\lambda'\left(su^++tu^-+\sigma\varepsilon\eta(s,t)\phi\right)\cdot\left(\varepsilon\eta(s,t)\phi\right) d\sigma\\
=&J_\lambda(su^++tu^-)+\varepsilon\eta(s,t)\int_0^1 J_\lambda'\left(su^++tu^-+\sigma\varepsilon\eta(s,t)\phi\right)\cdot\phi d\sigma\\
\leq&J_\lambda(su^++tu^-)-\frac{1}{2}\varepsilon\eta(s,t).
\end{aligned}\]
For the other case, that is $\vert s-1\vert\geq\varepsilon$ or $\vert t-1\vert\geq\varepsilon$, $\eta(s,t)=0$, the above estimate is obvious. Now since $u\in\mathcal{M}_\lambda$, for $(s,t)\neq(1,1)$, by Lemma \ref{lemJJst}, we have  $J_\lambda(su^++tu^-)<J_\lambda(u)$.
Hence
\[
J_\lambda\left(su^++tu^-+\varepsilon\eta(s,t)\phi\right)\leq J_\lambda(su^++tu^-)<J_\lambda(u) \ \mbox{for all } (s,t)\neq(1,1).
\]
For $(s,t)=(1,1)$,
\[
J_\lambda\left(su^++tu^-+\varepsilon\eta(s,t)\phi\right)\leq J_\lambda(su^++tu^-)-\frac{1}{2}\varepsilon\eta(1,1)=J_\lambda(u)-\frac{1}{2}\varepsilon.
\]
In any case, we have $J_\lambda\left(su^++tu^-+\varepsilon\eta(s,t)\phi\right)< J_\lambda(u)=m_\lambda$. In particular, for $0<\varepsilon<1-\varepsilon$,
\[
\sup\limits_{\varepsilon\leq s,t\leq2-\varepsilon}J_\lambda\left(su^++tu^-+\varepsilon\eta(s,t)\phi\right)=\widetilde{m}_\lambda<m_\lambda.
\]
Set $v=su^++tu^-+\varepsilon\eta(s,t)\phi$ and define
\[
H(s,t)=\left(F_1(s,t),F_2(s,t)\right)\dot{=}\left(J_\lambda'(v)\cdot v^+,J_\lambda'(v)\cdot v^-\right).
\]
By the definition of $\eta$, when $s=\varepsilon,\ t\in(\epsilon,2-\epsilon)$, we have $\eta(s,t)=0$ and $s<t$. Hence
\[\begin{aligned}
F_1(\varepsilon,t)\dot{=}&J'_\lambda(su^++tu^-)\cdot (su^+)\Big|_{s=\varepsilon}\\
=&\left[s^2J'_\lambda(u^+)\cdot u^+-\frac{st}{2}K_V(u)-s^2\log s^2\|u^+\|_2^2\right]_{s=\varepsilon}\\
>&\left[s^2\left(J'_\lambda(u^+)-\frac{1}{2}K_V(u)\right)-s^2\log s^2\|u^+\|_2^2\right]_{s=\varepsilon}\\
=&-\varepsilon^2\log\varepsilon^2\|u^+\|_2^2\\
>&0.
\end{aligned}\]
When $s=2-\varepsilon,\ t\in(\epsilon,2-\epsilon)$, we have $\eta(s,t)=0$ and $s>t$. Therefore,
\[\begin{aligned}
F_1(2-\varepsilon,t)\dot{=}&J'_\lambda(su^++tu^-)\cdot (su^+)\Big|_{s=2-\varepsilon}\\
=&\left[s^2J'_\lambda(u^+)\cdot u^+-\frac{st}{2}K_V(u)-s^2\log s^2\|u^+\|_2^2\right]_{s=2-\varepsilon}\\
<&\left[s^2\left(J'_\lambda(u^+)-\frac{1}{2}K_V(u)\right)-s^2\log s^2\|u^+\|_2^2\right]_{s=2-\varepsilon}\\
=&-(2-\varepsilon)^2\log(2-\varepsilon)^2\|u^+\|_2^2\\
<&0.
\end{aligned}\]
That is
\[\begin{aligned}
F_1(\varepsilon,t)>0,\ F_1(2-\varepsilon,t)<0 \ \mbox{for all }t\in(\varepsilon, 2-\varepsilon).
\end{aligned}\]
Similarly, we have
\[\begin{aligned}
F_2(s,\varepsilon)>0,\ F_2(s,2-\varepsilon)<0 \ \mbox{for all }s\in(\varepsilon, 2-\varepsilon).
\end{aligned}\]
Thus, applying Miranda's theorem \cite{M}, there exists $(s_0,t_0)\in(\varepsilon,2-\varepsilon)\times(\varepsilon,2-\varepsilon)$ such that $\widetilde{u}=s_0u^++t_0u^-+\varepsilon\eta(s_0,t_0)\phi\in\mathcal{M}_\lambda$ and $J_\lambda(\widetilde{u})<m_\lambda$. This give a contradiction to the definition of $m_\lambda$.

Next, we prove that $m_\lambda>2c_\lambda$. Assume that $u\in\mathcal{M}_\lambda$ such that $J_\lambda(u)=m_\lambda$. Then $u^\pm\neq0$. Similar to the proof of Lemma \ref{lememptyset} and Lemma \ref{lemstleq1}, we can deduce that there exists a unique $s_{u^+}\in(0,1]$ such that $s_{u^+}u^+\in\mathcal{N}_\lambda$, and a unique $t_{u^-}\in(0,1]$ such that $t_{u^-}u^-\in\mathcal{N}_\lambda$. Similar to the proofs of Lemma \ref{lemm-lambda>0} and Lemma \ref{lemsign-solu}, we can deduce that $c_\lambda>0$ can be achieved. Furthermore, if $u\in\mathcal{N}_\lambda$ with $J_\lambda(u)=c_\lambda$, then $u$ is a least energy solution.

By the definition of $J_\lambda$ and $K_V(x,y)<0$, we have
\[\begin{aligned}
J_\lambda(s_{u^+}u^++t_{u^-}u^-)=&J_\lambda(s_{u^+}u^+)+J_\lambda(t_{u^-}u^-)-\frac{s_{u^+}t_{u^-}}{2}K_V(u)\\
>&J_\lambda(s_{u^+}u^+)+J_\lambda(t_{u^-}u^-).
\end{aligned}\]
By Lemma \ref{lemJJst}, we deduce that
\[\begin{aligned}
m_\lambda=J_\lambda(u^++u^-)\geq J_\lambda(s_{u^+}u^++t_{u^-}u^-)>J_\lambda(s_{u^+}u^+)+J_\lambda(t_{u^-}u^-)\geq 2c_\lambda.
\end{aligned}\]
This completes the proof.
\end{proof}

\begin{proof}[Proof of Theorem \ref{th1.2}]\quad
The proof can be obtained by similar arguments in Theorem \ref{th1.1}.
\end{proof}

\section{Convergence of the least energy sign-changing solution $u_{\lambda}$}\label{Proof of Theorem 1.3}

In this section, we shall study the asymptotic behavior of $u_{\lambda}$ as $\lambda\to+\infty$. Firstly, we show that $\{u_{\lambda}\}$ is uniformly bounded above and below away from zero.

\begin{lemma}\label{lemunif-lowbdd}\quad
There exists $\sigma>0$ (independent of $\lambda$) such that $\|u\|_{\mathcal{H}_\lambda}\geq\|u\|_{H^1(V)}\geq\sigma$ for all $u\in\mathcal{M}_\lambda$.
\end{lemma}

\begin{proof}\quad
Note that for all $\varepsilon>0$, if $s\geq e^{-\frac{1}{2}}$, then
\begin{equation}\label{eqs}
e^{\frac{\varepsilon}{2}}s^{2+\varepsilon}\geq s^2.
\end{equation}
Since $u\in\mathcal{M}_\lambda$, by Lemma \ref{lemA2} and (\ref{eqs}), we have
\[\begin{aligned}
0=&J'_{\lambda}(u)\cdot u^+=J'_\lambda(u^+)\cdot u^+-\frac{1}{2}K_V(u)\\
\geq&\int_V\left(\Gamma(u^+)+(\lambda a(x)+1)|u^+|^2\right)d\mu-\int_V|u^+|^2d\mu-\int_V|u^+|^2\log|u^+|^2d\mu\\
=&\|u^+\|_{\mathcal{H}_{\lambda}}^2-\int_{|u^+|<e^{-\frac{1}{2}}}\left(|u^+|^2+|u^+|^2\log |u^+|^2\right)d\mu-\int_{|u^+|\geq e^{-\frac{1}{2}}}|u^+|^2d\mu\\
&-\int_{e^{-\frac{1}{2}}\leq|u^+|\leq1}|u^+|^2\log |u^+|^2d\mu-\int_{|u^+|>1}|u^+|^2\log |u^+|^2d\mu\\
\geq&\|u^+\|_{\mathcal{H}_{\lambda}}^2-e^{\frac{\varepsilon}{2}}\int_{|u^+|\geq e^{-\frac{1}{2}}}|u^+|^{2+\varepsilon}d\mu-C_{\varepsilon}\int_{|u^+|>1}|u^+|^{2+\varepsilon}d\mu\\
\geq&\|u^+\|_{\mathcal{H}_{\lambda}}^2-C_{\varepsilon}'\int_V|u^+|^{2+\varepsilon}d\mu\\
\geq&\|u^+\|_{H^1(V)}^2-C_{\varepsilon}''\|u^+\|_{H^1(V)}^{2+\varepsilon}.
\end{aligned}\]
Then
\[
\|u^+\|_{\mathcal{H}_\lambda}\geq \|u^+\|_{H^1(V)}\geq(C''_{\varepsilon})^{-\frac{1}{\varepsilon}}>0.
\]
Similarly, we get
\[
\|u^-\|_{\mathcal{H}_\lambda}\geq \|u^-\|_{H^1(V)}\geq(C''_{\varepsilon})^{-\frac{1}{\varepsilon}}>0.
\]
Hence,
\[\begin{aligned}
\|u\|_{\mathcal{H}_\lambda}^2\geq\|u\|_{H^1(V)}^2=\|u^+\|_{H^1(V)}^2+\|u^-\|_{H^1(V)}^2-K_V(u)>\|u^+\|_{H^1(V)}^2+\|u^-\|_{H^1(V)}^2\geq2(C''_{\varepsilon})^{-\frac{2}{\varepsilon}}.
\end{aligned}\]
Thus we can choose $\sigma=\sqrt{2}(C''_{\varepsilon})^{-\frac{1}{\varepsilon}}$ such that $\|u\|_{\mathcal{H}_\lambda}\geq\|u\|_{H^1(V)}\geq\sigma$.
\end{proof}

\begin{lemma}\label{lemunif-uppbdd}\quad
There exists $c_0>0$ (independent of $\lambda$) such that if sequence $\{u_k\}\subset\mathcal{M}_\lambda$ of $J_\lambda$ with $\lim_{k\to\infty}J_\lambda(u_k)=m_\lambda$, then $\|u_k\|_{\mathcal{H}_\lambda}\leq c_0$.
\end{lemma}

\begin{proof}\quad
Since $\mathcal{M}_{\Omega}\subset\mathcal{M}_{\lambda}$, it is easily seen that $m_{\lambda}\leq m_{\Omega}$ for any $\lambda>0$. Since $\{u_k\}\subset\mathcal{M}_\lambda$ and $\lim_{k\to\infty}J_\lambda(u_k)=m_\lambda$, we have
\begin{equation}\label{eqm-lambda'}
\begin{aligned}
\lim_{k\to+\infty}J_\lambda(u_k)=&\lim_{k\to+\infty}\left[J_\lambda(u_k)-\frac{1}{2}J'_\lambda(u_k)\cdot u_k^+-\frac{1}{2}J_\lambda'(u_k)\cdot u_k^-\right]\\
=&\lim_{k\to+\infty}\left(\frac{1}{2}\|u_k^+\|_2^2+\frac{1}{2}\|u_k^-\|_2^2\right)=m_\lambda\leq m_{\Omega}.
\end{aligned}
\end{equation}
By Lemma \ref{lemA2}, the H\"{o}lder's inequality and Young inequality, for any $\varepsilon\in(0,1)$, there exist $C_\varepsilon, C_\varepsilon', C_\varepsilon''>0$ such that
\[\begin{aligned}
\int_V|u_k^\pm|^2\log \vert u_k^\pm\vert^2d\mu\leq&\int_V(\vert u_k^\pm\vert^2\log \vert u_k^\pm\vert^2)^+d\mu\leq C_\varepsilon\int_V|u_k^\pm|^{2+\varepsilon}d\mu\\
\leq&C_\varepsilon\left(\int_V\vert u_k^\pm\vert^2d\mu\right)^{\frac{1}{2}}\left(\int_V\vert u_k^\pm\vert^{2(1+\varepsilon)}d\mu\right)^{\frac{1}{2}}\\
\leq&C_\varepsilon'\|u_k^\pm\|_2\|u_k^\pm\|_{\mathcal{H}_\lambda}^{1+\varepsilon}\\
\leq&\frac{1}{2}\|u_k^\pm\|_{\mathcal{H}_\lambda}^2+C_\varepsilon''\|u_k^\pm\|_2^{\frac{2}{1-\varepsilon}}.
\end{aligned}\]
Since $\{u_k\}\subset\mathcal{M}_\lambda$, we deduce that
\[\begin{aligned}
\|u_k^\pm\|_{\mathcal{H}_\lambda}^2-\frac{1}{2}K_V^k(x,y)\leq\frac{1}{2}\|u_k^\pm\|_{\mathcal{H}_\lambda}^2+C_\varepsilon''\|u_k^\pm\|_2^{\frac{2}{1-\varepsilon}}+\|u_k^\pm\|_2^2.
\end{aligned}\]
This together with (\ref{eqm-lambda'}) implies
\[\begin{aligned}
\lim_{k\to+\infty}\left(\|u_k^\pm\|_{\mathcal{H}_\lambda}^2-\frac{1}{2}K_V^k(u)\right)
\leq C_{\varepsilon}'''\left(m_\Omega^{\frac{1}{1-\varepsilon}}+m_\Omega\right).
\end{aligned}\]
From Lemma \ref{lemm-lambda>0} we know $m_\lambda>0$ and $m_\Omega>0$. Therefore it suffices to choose $c_0=C_{\varepsilon}'''\left(m_\Omega^{\frac{1}{1-\varepsilon}}+m_\Omega\right)$.
\end{proof}

Secondly, we have the following relation about $m_{\lambda}$ and $m_{\Omega}$.

\begin{lemma}\label{lemm-Ltom-O}\quad
$m_{\lambda}\to m_{\Omega}$ as $\lambda\to+\infty$.
\end{lemma}

\begin{proof}\quad
By $m_{\lambda}\leq m_{\Omega}$ for any $\lambda>0$, passing to subsequence if necessary, we may take a sequence $\lambda_k\to+\infty$ such that
\begin{equation}\label{eq4.1}
\lim_{k\to\infty}m_{\lambda_k}=\eta\leq m_{\Omega},
\end{equation}
where $m_{\lambda_k}=\inf\limits_{u_k\in\mathcal{M}_{\lambda_k}}J_{\lambda_k}(u_k)$. Then, combining Lemma \ref{lemunif-lowbdd} and \eqref{eq2}, it is easy to get $\eta>0$. By Lemma \ref{lemunif-uppbdd}, we have that $\{u_{\lambda_k}\}$ is uniformly bounded in $\mathcal{H}_{\lambda_k}$. Consequently, $\{u_{\lambda_k}\}$ is also bounded in $H^1(V)$ and thus, up to a subsequence, there exists some $u_0\in H^1(V)$ such that
\begin{equation}\label{eq4.2}
\begin{cases}
u_{\lambda_k}\rightharpoonup u_0 ~ \mbox{ weakly in }H^1(V),\\
u_{\lambda_k}\to u_0 ~ \mbox{ point-wisely in } V,\\
u_{\lambda_k}\to u_0 ~ \mbox{ strongly in } L^p(V) \mbox{ for } p\in[2,+\infty].
\end{cases}
\end{equation}
We claim that $u_0\mid_{\Omega^c}=0$. In fact, if there exists a vertex $x_0\in\Omega^c$ such that $u_0(x_0)\neq0$. Since $u_{\lambda_k}\in\mathcal{M}_{\lambda_k}$, we have
\[\begin{aligned}
J_{\lambda_k}(u_{\lambda_k})&=\frac{1}{2}\|u_{\lambda_k}\|_{\mathcal{H}_{\lambda_k}}^2-\frac{1}{2}\int_Vu_{\lambda_k}^2\log u_{\lambda_k}^2d\mu\\
&\geq\frac{\lambda_k}{2}\int_Va(x)u_{\lambda_k}^2d\mu-\frac{1}{2}\int_V(u_{\lambda_k}^2\log u_{\lambda_k}^2)^+d\mu\\
&\geq\frac{\lambda_k}{2}\int_Va(x)u_{\lambda_k}^2d\mu-\frac{C_{\varepsilon}}{2}\int_V|u_{\lambda_k}|^{2+\varepsilon}d\mu\\
&\geq\frac{\lambda_k}{2}\sum_{x\in V}\mu(x)a(x)u^2_{\lambda_k}(x)-C'_{\varepsilon}\|u_{\lambda_k}\|_{H^1(V)}^{2+\varepsilon}\\
&\geq\frac{\lambda_k}{2}\mu_{\min}a(x_0)u^2_{\lambda_k}(x_0)-C''_{\varepsilon}.
\end{aligned}\]
Since $a(x_0)>0$, $u_{\lambda_k}(x_0)\to u_0(x_0)\neq0$ and $\lambda_k\to+\infty$, we get
\[
\lim_{k\to+\infty}J_{\lambda_k}(u_{\lambda_k})=+\infty,
\]
This is in contradiction with (\ref{eq4.1}). Hence the claim holds.

Since $u_0\mid_{\Omega^c}=0$, by the weak lower semi-continuity of the norm $\|\cdot\|_{H^1(V)}$ and Fatou's lemma, taking $u_{\lambda_k}^+$ as test function in equation (\ref{eq-eq1}), we get
\[\begin{aligned}
&\int_{\Omega\cup\partial\Omega}\Gamma(u_0^+)d\mu+\int_{\Omega}|u_0^+|^2d\mu-\int_{\{\Omega:|u_0^+|\leq1\}}|u_0^+|^2\log|u_0^+|^2d\mu-\frac{1}{2}K_\Omega^0(u)\\
\leq&\int_V\left(\Gamma(u_0^+)+|u_0^+|^2\right)d\mu-\int_{\{V:|u_0^+|\leq1\}}|u_0^+|^2\log|u_0^+|^2d\mu-\frac{1}{2}K_V^0(u)\\
\leq&\liminf_{k\to+\infty}\left[\int_V\left(\Gamma(u_{\lambda_k}^+)+\left(\lambda_ka(x)+1\right)|u_{\lambda_k}^+|^2\right)d\mu-\int_{\{V:|u_{\lambda_k}^+|\leq1\}}|u_{\lambda_k}^+|^2\log|u_{\lambda_k}^+|^2d\mu-\frac{1}{2}K_V^{\lambda_k}(u)\right]\\
=&\liminf_{k\to+\infty}\left[\int_V|u_{\lambda_k}^+|^2d\mu+\int_{\{V:|u_{\lambda_k}^+|>1\}}|u_{\lambda_k}^+|^2\log|u_{\lambda_k}^+|^2d\mu\right]\\
=&\int_V|u_0^+|^2d\mu+\int_{\{V:|u_0^+|>1\}}|u_0^+|^2\log|u_0^+|^2d\mu\\
=&\int_{\Omega}|u_0^+|^2d\mu+\int_{\{\Omega:|u_0^+|>1\}}|u_0^+|^2\log|u_0^+|^2d\mu,
\end{aligned}\]
where
\begin{gather*}
K_\Omega^0(u)=\sum_{x\in\Omega\cup\partial\Omega}\sum_{y\sim x}\left[u_0^+(x)u_0^-(y)+u_0^-(x)u_0^+(y)\right],\\
K_V^0(u)=\sum_{x\in V}\sum_{y\sim x}\left[u_0^+(x)u_0^-(y)+u_0^-(x)u_0^+(y)\right],\\
K_V^{\lambda_k}(u)=\sum_{x\in V}\sum_{y\sim x}\left[u_{\lambda_k}^+(x)u_{\lambda_k}^-(y)+u_{\lambda_k}^-(x)u_{\lambda_k}^+(y)\right].
\end{gather*}
Then
\begin{equation*}\label{equ00+}
J'_\Omega(u_0)\cdot u_0^+=\int_{\Omega\cup\partial\Omega}\Gamma(u_0^+)d\mu-\int_{\Omega}|u_0^+|^2\log|u_0^+|^2d\mu-\frac{1}{2}K_\Omega^0(u)\leq0.
\end{equation*}
Similarly, it holds that
\begin{equation*}\label{equ00-}
J'_\Omega(u_0)\cdot u_0^-=\int_{\Omega\cup\partial\Omega}\Gamma(u_0^-)d\mu-\int_{\Omega}|u_0^-|^2\log|u_0^-|^2d\mu-\frac{1}{2}K_\Omega^0(u)\leq0.
\end{equation*}
In view of Lemma \ref{coremptyset} and Lemma \ref{corstleq1}, there exist $s, t\in(0,1]$ such that $\widetilde{u}_0=su_0^++tu_0^-\in\mathcal{M}_{\Omega}$. Then
\[\begin{aligned}
m_{\Omega}\leq&J_{\Omega}(\widetilde{u}_0)=J_{\Omega}(\widetilde{u}_0)-\frac{1}{2}J'_{\Omega}(\widetilde{u}_0)\cdot(su_0^+)-\frac{1}{2}J'_{\Omega}(\widetilde{u}_0)\cdot(tu_0^-)\\
=&\frac{s^2}{2}\|u_0^+\|_{L^2(\Omega)}^2+\frac{t^2}{2}\|u_0^-\|_{L^2(\Omega)}^2\\
\leq&\liminf_{k\to\infty}\left[\frac{1}{2}\|u_{\lambda_k}^+\|_2^2+\frac{1}{2}\|u_{\lambda_k}^-\|_2^2\right]\\
=&\liminf_{k\to\infty}\left[J_{\lambda_k}(u_{\lambda_k})-\frac{1}{2}J'_{\lambda_k}(u_{\lambda_k})\cdot u_{\lambda_k}^+-\frac{1}{2}J'_{\lambda_k}(u_{\lambda_k})\cdot u_{\lambda_k}^-\right]\\
=&\liminf_{k\to+\infty}J_{\lambda_k}(u_{\lambda_k})=\eta\le m_{\Omega}.
\end{aligned}\]
Hence, $\lim_{\lambda\to+\infty}m_{\lambda}=m_{\Omega}$.
This completes the proof.
\end{proof}

\begin{proof}[Proof of Theorem \ref{th1.3}]\quad
Assume that $u_{\lambda_k}\in\mathcal{M}_{\lambda_k}$ satisfies $J_{\lambda_k}(u_{\lambda_k})=m_{\lambda_k}$. We shall prove that $u_{\lambda_k}$ converges in $H^1(V)$ to a least energy sign-changing solution $u_0$ of equation $(\ref{eq-eq2})$ along a subsequence.

Lemma \ref{lemunif-uppbdd} gives that $u_{\lambda_k}\in\mathcal{H}_{\lambda_k}$ is uniformly bounded. Consequently, we have that $\{u_{\lambda_k}\}$ is also bounded in $H^1(V)$. Therefore, we can assume that for any $p\in[2,\infty)$, $u_{\lambda_k}\to u_0$ in $L^p(V)$ and $u_{\lambda_k}\rightharpoonup u_0$ in $H^1(V)$. Moreover, in view of $u_0\in\mathcal{N}_\Omega$ and we get from Lemma \ref{lemunif-lowbdd} that $u_0\not\equiv0$. As proved in Lemma \ref{lemm-Ltom-O}, we can prove that $u_0\mid_{\Omega^c}=0$. Then it suffices to show that, as $k\to+\infty$, we have $\lambda_k\int_Va(x)\vert u_{\lambda_k}^\pm\vert^2d\mu\to0$ and $\int_V\Gamma(u_{\lambda_k}^\pm)d\mu\to\int_V\Gamma(u_0^\pm)d\mu$.
If not, we may assume that
\[
\lim\limits_{k\to+\infty}\lambda_k\int_Va(x)\vert u_{\lambda_k}^\pm\vert^2d\mu=\delta>0.
\]
Since $u_0\mid_{\Omega^c}=0$, by weak lower semi-continuity of the norm $\|\cdot\|_{H^1(V)}$ and Fatou's lemma, taking $u_{\lambda_k}^+$ as test function in equation (\ref{eq-eq1}), we get
\[\begin{aligned}
&\int_{\Omega\cup\partial\Omega}\Gamma(u_0^+)d\mu+\int_{\Omega}|u_0^+|^2d\mu-\int_{\{\Omega:|u_0^+|\leq1\}}|u_0^+|^2\log|u_0^+|^2d\mu-\frac{1}{2}K_\Omega^0(u)\\
<&\int_V\left(\Gamma(u_0^+)+|u_0^+|^2\right)d\mu+\delta-\int_{\{V:|u_0^+|\leq1\}}|u_0^+|^2\log|u_0^+|^2d\mu-\frac{1}{2}K_V^0(u)\\
\leq&\liminf_{k\to+\infty}\left[\int_V\left(\Gamma(u_{\lambda_k}^+)+\left(\lambda_ka(x)+1\right)|u_{\lambda_k}^+|^2\right)d\mu-\int_{\{V:|u_{\lambda_k}^+|\leq1\}}|u_{\lambda_k}^+|^2\log|u_{\lambda_k}^+|^2d\mu-\frac{1}{2}K_V^{\lambda_k}(u)\right]\\
=&\liminf_{k\to+\infty}\left[\int_V|u_{\lambda_k}^+|^2d\mu+\int_{\{V:|u_{\lambda_k}^+|>1\}}|u_{\lambda_k}^+|^2\log|u_{\lambda_k}^+|^2d\mu\right]\\
=&\int_V|u_0^+|^2d\mu+\int_{\{V:|u_0^+|>1\}}|u_0^+|^2\log|u_0^+|^2d\mu\\
=&\int_{\Omega}|u_0^+|^2d\mu+\int_{\{\Omega:|u_0^+|>1\}}|u_0^+|^2\log|u_0^+|^2d\mu,
\end{aligned}\]
which implies that
\begin{equation}\label{equ000+}
J'_\Omega(u_0)\cdot u_0^+=\int_{\Omega\cup\partial\Omega}\Gamma(u_0^+)d\mu-\int_{\Omega}|u_0^+|^2\log|u_0^+|^2d\mu-\frac{1}{2}K_\Omega^0(u)<0.
\end{equation}
Similarly, it holds that
\begin{equation}\label{equ000-}
J'_\Omega(u_0)\cdot u_0^-=\int_{\Omega\cup\partial\Omega}\Gamma(u_0^-)d\mu-\int_{\Omega}|u_0^-|^2\log|u_0^-|^2d\mu-\frac{1}{2}K_\Omega^0(u)<0.
\end{equation}
By similar arguments as above, if
\[
\lim\limits_{k\to+\infty}\int_V\Gamma(u_{\lambda_k}^\pm)d\mu>\int_V\Gamma(u_0^\pm)d\mu,
\]
we also have (\ref{equ000+}) and (\ref{equ000-}).

In view of Lemma \ref{coremptyset} and Lemma \ref{corstleq1}, there exist two constants $s, t\in(0,1)$ such that $\widetilde{u}_0=su_0^++tu_0^-\in\mathcal{M}_{\Omega}$. Consequently, we have
\[\begin{aligned}
m_{\Omega}\leq&J_{\Omega}(\widetilde{u}_0)=J_{\Omega}(\widetilde{u}_0)-\frac{1}{2}J'_{\Omega}(\widetilde{u}_0)\cdot (su_0^+)-\frac{1}{2}J'_{\Omega}(\widetilde{u}_0)\cdot (tu_0^-)\\
=&\frac{s^2}{2}\|u_0^+\|_{L^2(\Omega)}^2+\frac{t^2}{2}\|u_0^-\|_{L^2(\Omega)}^2\\
<&\frac{1}{2}\|u_0^+\|_2^2+\frac{1}{2}\|u_0^-\|_2^2\\
\leq&\liminf_{k\to+\infty}\left[\frac{1}{2}\|u_{\lambda_k}^+\|_2^2+\frac{1}{2}\|u_{\lambda_k}^-\|_2^2\right]\\
=&\liminf_{k\to+\infty}\left[J_{\lambda_k}(u_{\lambda_k})-\frac{1}{2}J'_{\lambda_k}(u_{\lambda_k})\cdot u_{\lambda_k}^+-\frac{1}{2}J'_{\lambda_k}(u_{\lambda_k})\cdot u_{\lambda_k}^-\right]\\
=&\liminf_{k\to+\infty}J_{\lambda_k}(u_{\lambda_k})\\
=&\liminf_{k\to+\infty}m_{\lambda_k}=m_{\Omega},
\end{aligned}\]
which leads to a contradiction. Hence, we obtain that $u_{\lambda_k}\to u_0$ in $H^1(V)$ and $u_0$ is a least energy sign-changing solution of problem (\ref{eq-eq2}).
\end{proof}

\section*{Declaration of competing interest}
The authors declare that they have no known competing financial interests or personal relationships that could have appeared to influence the work reported in this paper.

\section*{Data availability}
No data was used for the research described in the article.

\section*{Acknowledgements}
The research of Xiaojun Chang was supported by National Natural Science Foundation of China (Grant No.11971095), while Duokui Yan was supported by National Natural Science Foundation of China (Grant No.11871086).

\section*{References}
\biboptions{sort&compress}

\end{document}